\documentclass[reqno,12pt,a4letter]{amsart}
\usepackage{amsmath, amsxtra, amssymb, latexsym, amscd, amsthm,enumerate}
\usepackage{graphicx, color}
\usepackage[active]{srcltx}
\usepackage[utf8]{inputenc}
\usepackage[mathscr]{euscript}
\usepackage{mathrsfs,cite}
\usepackage[english]{babel}
\usepackage{color}
\usepackage{hyperref}
\usepackage[active]{srcltx}
\def\ri{\mbox{\rm ri}\,}

\def\Int{\mbox{\rm int}\,}
\def\gph{\mbox{\rm gph}\,}
\def\epi{\mbox{\rm epi}\,}

\def\dom{\mbox{\rm dom}\,}

\def\aff{\mbox{\rm aff}\,}
\def\R{\mathbb{R}}
\def\N{\mathbb{N}}

\newtheorem {theorem}{Theorem}[section]
\newtheorem {corollary}{Corollary}[section]
\newtheorem {proposition}{Proposition}[section]

\newtheorem {example}{Example}[section]
\newtheorem {definition}{Definition}[section]
\newtheorem {remark}{Remark}[section]

\title{Approximate optimality conditions and  sensitivity analysis in nearly convex optimization}

\author{Nguyen Van Tuyen$^{1}$, Liguo Jiao$^{2}$, Vu Hong Quan$^{3}$, Duong Thi Viet An$^{4}$}
\address{$^{1}$Department of Mathematics, Hanoi Pedagogical University 2, Xuan Hoa, Phuc Yen, Vinh Phuc, Vietnam}
\email{nguyenvantuyen83@hpu2.edu.vn; tuyensp2@yahoo.com}
\address{$^{2}$Academy for Advanced Interdisciplinary Studies, Northeast Normal University, Changchun, 130024, Jilin Province, China}
\email{jiaolg356@nenu.edu.cn}
\address{$^{3}$Faculty of Fundamental and Applied Sciences, Thai Nguyen University of Technology, Thai Nguyen 250000, Vietnam}
\email{hongquan@tnut.edu.vn}
\address{$^4$Department of Mathematics and Informatics, Thai Nguyen University of Sciences, Thai Nguyen 250000, Vietnam
}
\email{andtv@tnus.edu.vn}

\date{\today}



\begin{document}	
	\maketitle
	\begin{abstract}
		In this paper, approximate optimality conditions and sensitivity analysis in nearly convex optimization are discussed. More precisely,   as in the spirit of convex analysis, we introduce the concept of $\varepsilon$-subdifferential for nearly convex functions.  Then, we examine some significant properties and rules for the $\varepsilon$-subdifferential. These rules are applied to study optimality conditions as well as sensitivity analysis for parametric nearly convex optimization problems, which are two important topics in optimization theory.     
	\end{abstract}
	
	\section{Introduction}
	
	In the 1960s, Minty and Rockafellar introduced nearly convex sets~\cite{Minty_61},~\cite{Rockafellar_1970}. The nearly convex set is defined based on convexity by requiring that the set under consideration lies between a convex set and its closure in the Euclidean space $\mathbb{R}^n$. The natural reason to study nearly convex sets is that the domain and the range of any maximal monotone mapping are nearly convex. In particular, for a proper lower semicontinuous convex function its subdifferential domain is always nearly convex~\cite[Theorem 12.41]{Rockafellar_Wet}. Moreover, the classical notion of convexity and convex analysis has been extensively studied by prominent mathematicians and experts in applied fields. Another motivation is that researchers want to beyond convexity by introducing and studying many generalized convexity notions for sets and functions. All these motivate the researcher's systematic study of nearly convex sets. Some properties of nearly convex sets have been partially studied in~\cite{Bauschke_Moffat_Wang}, \cite{Bot_grad_Wanka_06}, \cite{Bot_grad_Wanka_07},\cite{Moffat_etal_2016}, and~\cite{Nam_Thieu_Yen_23} from different perspectives.
	
	\medskip
	
	Very recently, Nam and his co-workers in~\cite{Nam_Thieu_Yen_23} introduced the concept of nearly convex set-valued mappings and investigated the fundamental properties of these mappings. Additionally, the authors establish a geometric approach for generalized differentiation of nearly convex set-valued mappings and nearly convex functions. These contributions expand the current knowledge of nearly convex sets and functions while providing several new results involving nearly convex set-valued mappings. Despite being introduced in the early age of convex analysis, the notion of the near convexity had not been systematically studied in the literature. The new development also creates opportunities for further study, from nearly convex sets to nearly convex functions and nearly convex set-valued mappings.
	
	\medskip
	
	To deal with optimization problems, one often uses optimality conditions, in particular, necessary optimality conditions. Necessary optimality conditions help us solve problems through manual calculations and are useful as stopping criteria in algorithms. Meanwhile, sensitivity analysis of parametric optimization problems is not only theoretically interesting but also critically important. It allows us to understand the behaviors of the optimal value function when the parameter of the problem undergoes perturbations. Many researchers have made contributions to these research directions, for example, An and Yen~\cite{An_Yen}, An and co-authors~\cite{An_Markus-Tuyen-20}, Ioffe and Penot~\cite{Ioffe_Penot}, Jourani~\cite{Jourani}, Mordukhovich and co-authors~\cite{Mordukhovich_etal}. It is well recognized that qualification conditions are sufficient conditions for the validity of fundamental calculus rules in variational analysis, convex analysis, and optimization theory. These conditions play vital roles in deriving intersection rules for normal cones. The latter rule plays an important role in the study of optimality conditions and sensitivity analysis of parametric optimization problems. Therefore, one of the biggest challenges is to apply the subdifferential sum rules with the weakest qualification conditions.
	
	The concept of $\varepsilon$-\textit{subdifferentials} (known also as \textit{approximate subdifferentials}) for convex functions was introduced by Br\o{n}dsted and Rockafellar in~\cite{Brondsted_Rockafellar} who proved there the Br\o{n}dsted-Rockafellar density theorem and established other topological properties of $\varepsilon$-subdifferentials.  It has been realized later on that $\varepsilon$-subgradient mappings for $\varepsilon>0$ exhibit some better properties in comparison with the classical case of $\varepsilon=0.$ This made it possible to use $\varepsilon$-subgradients in constructing efficient numerical algorithms, which were started from the paper by Bertsekas and Mitter~\cite{Bertsekas_Mitter_71,Bertsekas_Mitter_73}. For more information, the reader is referred to~\cite{Dhara-Dutta-12,Hiriart_1982,H_L_1993,H_M_S,Tuyen-et-al-20,Zalinescu_2002} and the references therein.
	
	\medskip
	In the present paper,  as in the spirit of convex analysis, we introduce the concept of the $\varepsilon$-subdifferential for nearly convex functions.  We prove some significant properties and rules related to it. Then we use the $\varepsilon$-subdifferential to study optimality conditions as well as sensitivity analysis for parametric nearly convex optimization problems, which are two important topics in optimization theory. 
	
	\medskip
	
	Our paper is organized as follows. In Section 2, we provide an overview of the definition and properties of the $\varepsilon$- subdifferential of nearly convex functions. In Section 3, we present optimality conditions for nearly convex optimization problems. Sensitivity analysis for parametric nearly convex optimization problems is discussed in Section 4.
	
	\medskip
	
	Throughout the paper, the  space $\R^n$ is equipped with the usual scalar product $\langle \cdot, \cdot\rangle$ and the corresponding  Euclidean norm $\|\cdot\|$. We use the notation $B(x; \delta)$  (resp., $\mathbb{B}(x;\delta)$) to represent the open (resp., closed) ball centered at $x\in \mathbb{R}^n$ with a radius of $\delta>0.$  The interior (resp., closure) of $D$ will be denoted by $\Int D$ (resp., $\overline{D}$).
	Here $\aff D$ stands for the affine hull of a subset $D\subset \mathbb{R}^n$.
	\section{$\varepsilon$-subdifferential of nearly convex functions}
	In this section, we introduce the concept $\varepsilon$-subdifferential for nearly convex functions. Then we explore some properties of  this subdifferential.
	
	\medskip
	Let us first recall that the \textit{relative interior} of an arbitrary set $D$ in $\mathbb{R}^n$ is defined by 
	$$\ri D:=\{a\in D \mid \mbox{there exists  }\delta>0 \ \mbox{such that  } B(a;\delta) \cap \aff D \subset D\}.
	$$
By definition,	it is easy to see that $a \in \ri D$ if and only if $a \in \aff D $ and there exists
	$\delta>0$ such~that
	$$B(a;\delta) \cap \aff D\subset D.$$
	\begin{definition}\rm 
		A subset $D$ of $\mathbb{R}^n$ is said to be \textit{nearly convex} if there exists a convex set $E$ in $\mathbb{R}^n$ such that $E\subset D \subset \overline E.$   
	\end{definition}

	Clearly, any convex set is nearly convex and any nearly convex subset of $\mathbb R$ is convex. However, in $\mathbb{R}^n$ with $n\ge 2$ there are many nearly convex sets which are not convex,  see, for example \cite{Moffat_etal_2016}.

	It is not hard to check that if $D_1 \subset \mathbb{R}^n$ and $D_2 \subset \mathbb{R}^m$  are nearly convex, then $D_1 \times D_2$ is nearly convex as well.
	
	\medskip
	Consider a function  $\varphi:\mathbb{R}^n \rightarrow \overline{\mathbb{R}}$ having values in the extended real line $\overline{\mathbb{R}}:=[- \infty,  \infty]$. One says that $\varphi$ is \textit{proper} if the \textit{domain}	${\rm{dom}}\, \varphi:=\{ x \in \mathbb{R}^n \mid \varphi(x) < \infty\}$ is nonempty and if $\varphi(x) > - \infty$ for all $x \in \mathbb{R}^n$. The set $ {\rm{epi}}\, \varphi:=\{ (x, \alpha) \in \mathbb{R}^n \times \mathbb{R} \mid \alpha \ge \varphi(x)\}$ is called the \textit{epigraph} of $\varphi$. If ${\rm{epi}}\, \varphi$ is a convex (resp., nearly convex) subset of $\mathbb{R}^n\times {\mathbb{R}}$, then $\varphi$ is said to be a \textit{convex} (resp., {\it nearly convex}) function.
	Due to the continuity of a convex function $\varphi$ in $\ri(\dom \varphi),$ it knows that the ``near" convexity only happens in the boundary of its domain.

 \medskip
	Let  $G : \R^n \rightrightarrows \R^m$ be a set-valued mapping. The \textit{domain} and the \textit{graph}  of $G$ are given, respectively, by
	$${\rm dom}\,G:=\{x\in \R^n \mid G(x)\not= \emptyset\} $$
	and
	$$ {\rm gph}\,G:=\{(x,y)\in \R^n \times \R^m \mid y \in G(x)\}.$$
	The set-valued mapping $G$ is called  \textit{proper} if $\dom G \not= \emptyset.$ We say that $G$ is \textit{nearly convex}
	if its graph is a nearly convex set in $\R^n \times \R^m$.
	
	
	\begin{definition} \rm 
		Consider a proper   function $\varphi\colon \R^n\to\overline{\R}$. Let $\varepsilon\geq 0$ and $\bar x\in\dom \varphi$. The $\varepsilon$\textit{-subdifferential }of $\varphi$ at $\bar x$ is defined by
		\begin{equation*}
			\partial_\varepsilon \varphi(\bar x):=\{\xi\in\R^n \mid  \langle\xi, x-\bar x\rangle-\varepsilon\leq \varphi(x)-\varphi(\bar x),\ \ \forall x\in \R^n\}.
		\end{equation*}
	\end{definition}
	The set $\partial_\varepsilon \varphi(\bar x)$ reduces to the subdifferential $\partial \varphi(\bar x)$ when $\varepsilon=0.$ We will study some properties of the $\varepsilon$-subdifferential of nearly convex functions later. 
 
The following example shows that the traditional subdifferential $\partial \varphi(\bar x)$ may be empty, meanwhile for all $\varepsilon > 0$, the $\varepsilon$-subdifferential is nonempty. 
	\begin{example}\label{ex1}\rm 
		Consider the function $\varphi: \mathbb{R} \to \overline{\mathbb{R}}$ given by
		$$ \varphi(x)= \begin{cases} -\sqrt{x} & \mbox{if }\ \ x\in [0, 1),
			\\
			1& \mbox{if}\ \ x=1,
			\\
			\infty & \mbox{otherwise,}
		\end{cases}
		$$
		which is indeed a nearly convex function.  
		
		Let $\bar{x} = 0$ and for every $\varepsilon > 0$, by definition, one has
		\begin{align*}
			\partial_\varepsilon \varphi(\bar x)&= \{\xi \in \mathbb{R} \mid \langle \xi, x-\bar x \rangle \le \varphi(x)-\varphi(\bar x) + \varepsilon, \ \forall x\in \mathbb{R} \}\\
			&=\{\xi \in \mathbb{R} \mid  \xi x  \le \varphi(x) + \varepsilon, \ \forall x\in \mathbb{R} \}\\  
			&= \left(-\infty, -\frac{1}{4 \varepsilon} \right].
		\end{align*}
   However, it is easy to check that $\partial \varphi(\bar x)=\emptyset$. 
	\end{example}
	\begin{definition}
		\rm Let $\varepsilon\geq 0$, $\Omega$ be a nonempty subset in $\R^n$ and $\bar x\in\Omega$. The {\em $\varepsilon$-normal set} to $\Omega$ at $\bar x$ is defined by
		\begin{equation*}
			N_\varepsilon(\bar x; \Omega):=\{\xi\in\R^n\mid \langle \xi, x-\bar x\rangle\leq \varepsilon, \ \ \forall x\in\Omega\}.
		\end{equation*}
	\end{definition}
	By definition, it is clear that $N_\varepsilon(\bar x; \Omega)=\partial_\varepsilon\delta_\Omega(\bar x)$, where $\delta_\Omega$ is the {\it indicator function} of a set $\Omega \subset \mathbb{R}^n$. Recall that $\delta_\Omega (x)=0$ if $x \in \Omega$ and $ \delta_\Omega (x)=\infty$ if $x \notin \Omega$. Since $\epi\delta_\Omega = \Omega \times [0, \infty)$, we see that $\delta_\Omega$ is nearly convex if and only if $\Omega$ is nearly convex.
	
	When $\varepsilon = 0,$  $N_\varepsilon(\bar x;\Omega)$  reduces to the normal cone of $\Omega$ at $\bar x$, which is denoted by $N(\bar x;\Omega).$  However, in general $N_\varepsilon(\bar x;\Omega)$ is not a cone when $\varepsilon>0.$
	From the definition of the $\varepsilon$-normal set, one can define the $\varepsilon$-coderivative as follows.
	\begin{definition}
		\rm Let $F\colon \R^n\rightrightarrows\R^m$ be a nearly convex set-valued mapping. For each $\varepsilon\geq 0$, the $\varepsilon$-\textit{coderivative }of $F$ at $(\bar x, \bar y)\in\gph F$ is the set-valued mapping $D^*_{\varepsilon}F(\bar x, \bar y)\colon\R^m\rightrightarrows\R^n$ with the values
		\begin{equation*}
			D^*_{\varepsilon}F(\bar x, \bar y)(v):=\{u\in \mathbb{R}^n \mid (u, -v)\in N_\varepsilon ((\bar x, \bar y); \gph F)\},  \ \ v\in\R^m.
		\end{equation*}
	\end{definition}
	The following proposition gives the relation between the $\varepsilon$-subdifferential and the $\varepsilon$-coderivative of its epigraphical mapping.
	\begin{proposition}
		Let $\varphi\colon \R^n\to\overline{\R}$ be a   nearly convex function and let $\bar x\in \R^n$ be such that $\varphi(\bar x)\in\R$. For $\varepsilon\geq 0$, we have
		\begin{equation*}
			\partial_\varepsilon\varphi(\bar x)= D^*_{\varepsilon}E_\varphi(\bar x, \varphi(\bar x))(1),
		\end{equation*}
		where $E_\varphi : \R^n \rightrightarrows \R$ is the epigraphical mapping given by
		$$E_\varphi(x)=\{\lambda \in \mathbb R \mid \varphi(x) \le \lambda\},\  x\in \mathbb R^n.$$
	\end{proposition}
	\begin{proof}
		The proof is similar adapted from~\cite[Proposition 5.3]{Nam_Thieu_Yen_23}. For $\varepsilon \ge 0$, we first show that \begin{align}\label{inclu} D^*_{\varepsilon}E_\varphi(\bar x, \varphi(\bar x))(1) \subseteq \partial_\varepsilon\varphi(\bar x).   
		\end{align}
		By the definition of the $\varepsilon$-coderivative, one has
		\begin{align*} D^*_{\varepsilon}E_\varphi(\bar x, \varphi(\bar x))(1)&=\{u \in \mathbb{R}^n \mid (u, -1) \in N_\varepsilon ((\bar x, \varphi(\bar x)); \gph E_\varphi) \} \\
			&= \{u \in \mathbb{R}^n \mid (u, -1) \in N_\varepsilon ((\bar x, \varphi(\bar x)); \epi\varphi) \}. 
		\end{align*}
		Since $D^*_{\varepsilon}E_\varphi(\bar x, \varphi(\bar x))(1)=\emptyset$, the inclusion~\eqref{inclu} is obvious.
		We now pick an arbitrary vector $u\in D^*_{\varepsilon}E_\varphi(\bar x, \varphi(\bar x))(1)$, then
		$$ \langle u, x-\bar x \rangle - (\lambda -\varphi(\bar x)) \le \varepsilon, \ \forall(x, \lambda) \in \epi \varphi.$$
It follows from the above inequality that $-\infty < \varphi(x)$ for any $x\in \mathbb{R}^n$.   Using this equality with $x\in \dom \varphi$ and $\lambda=\varphi(x)$, one has
		$$\langle u, x-\bar x \rangle \le \varphi(x)-\varphi (\bar x) +\varepsilon, \ \forall x\in \dom \varphi.$$
		In the case, where $\varphi(x)=\infty$, the above inequality holds. So, from the latter, we get $u\in \partial_\varepsilon \varphi(\bar x)$, which justifies~\eqref{inclu}.
		
		To obtain the opposite inclusion, we take any $u\in \partial_\varepsilon \varphi(\bar x)$. Then
		$$\langle u, x-\bar x \rangle \le \varphi(x) -\varphi (\bar x) + \varepsilon, \ \forall x\in \mathbb{R}^n.$$
		This implies that 
		$$ \langle u, x-\bar x \rangle - (\lambda -\varphi(\bar x)) \le \varepsilon, \ \mbox{whenever }(x, \lambda) \in \epi \varphi.$$
		Thus, $\partial_\varepsilon \varphi(\bar x)\subset D^*_{\varepsilon}E_\varphi(\bar x, \varphi(\bar x))(1).$
		If $\partial_\varepsilon \varphi(\bar x)=\emptyset$, then the latter is obvious. Therefore, we conclude that 
		$\partial_\varepsilon\varphi(\bar x)= D^*_{\varepsilon}E_\varphi(\bar x, \varphi(\bar x))(1).$
	\end{proof}
	\medskip
	
	We now present some properties of the set $\varepsilon$-subdifferentials of nearly convex functions.
	
	\begin{proposition} Let $\varphi\colon \R^n\to\overline{\R}$ be a proper nearly convex function and $\varepsilon\geq 0$. Then, for every $\bar x\in \ri(\dom \varphi)$, the $\varepsilon$-subdifferential of $\varphi$ at $\bar x$ is nonempty, closed and convex. Furthermore, 
		\begin{equation}\label{equa_0}
			\partial \varphi(\bar x)=\bigcap_{\varepsilon>0} \partial_\varepsilon \varphi(\bar x).
		\end{equation} 
	\end{proposition}
	\begin{proof} We first prove that $\partial_\varepsilon \varphi(\bar x)$ is nonempty for every $\bar x\in \ri(\dom \varphi)$ and $\varepsilon\geq 0$. The nonemptiness of $\partial_\varepsilon \varphi(\bar x)$ when $\varepsilon=0$ was already proved in \cite[Proposition 5.4(b)]{Nam_Thieu_Yen_23}. We now consider the case that $\varepsilon>0$.    By the near convexity of $\varphi$ and \cite[Proposition 3.7]{Nam_Thieu_Yen_23}, one has
		\begin{equation*}
			\ri(\epi \varphi)=\{(x, \lambda)\in\R^n\times\R \mid x\in\ri(\dom \varphi), \ \varphi(x)<\lambda\}.
		\end{equation*} 
		Hence,  $\varphi(\bar x)>\varphi(\bar x)-\varepsilon$ and so $(\bar x, \varphi(\bar x)-\varepsilon)\notin \ri(\epi \varphi)$. By \cite[Theorem 3.2]{Nam_Thieu_Yen_23}, $\{(\bar x, \varphi(\bar x)-\varepsilon)\}$ and $\epi \varphi$ can be properly separated, i.e., there exist $(w, \beta)\in\R^n\times\R$ and $(x_0, \lambda_0)\in \epi \varphi$ such that the following conditions hold:
		\begin{align}
			&\langle w, \bar x\rangle+ \beta(\varphi(\bar x)-\varepsilon)\leq \langle w, x\rangle+\beta.\lambda \ \ \forall  (x, \lambda)\in \epi \varphi;\label{equa-1}
			\\
			&\langle w, \bar x\rangle+ \beta(\varphi(\bar x)-\varepsilon)< \langle w, x_0\rangle+\beta.\lambda_0.\label{equa-2}
		\end{align}
		It follows from \eqref{equa-1} that 
		\begin{equation}\label{equa-3}
			\langle w, \bar x\rangle+ \beta(\varphi(\bar x)-\varepsilon)\leq \langle w, x\rangle+\beta \varphi(x) \ \ \forall x\in\dom \varphi. 
		\end{equation}
		Taking $x=\bar x$ in the inequality \eqref{equa-3} yields $\beta\geq 0$. If $\beta=0$, then from~\eqref{equa-1} and~\eqref{equa-2}, one has $\langle w, \bar x\rangle\leq \langle w, x\rangle$ for all $x\in \dom \varphi$ and $\langle w, \bar x\rangle< \langle w, x_0\rangle$. This means that $\{\bar x\}$ and $\dom \varphi$ can be properly separated. By \cite[Theorem 3.2]{Nam_Thieu_Yen_23}, $\bar x\notin\ri(\dom \varphi)$, a contradiction. Therefore, $\beta>0$.  Dividing both sides of~\eqref{equa-3} by $\beta$, we get
		\begin{equation*}
			\left\langle-\tfrac{w}{\beta}, x-\bar x\right\rangle-\varepsilon\leq \varphi(x)-\varphi(\bar x), \ \ \forall x\in\dom \varphi.
		\end{equation*}
		Clearly, the above condition always holds for $x\notin\dom \varphi$. So, $-\tfrac{w}{\beta}\in \partial_\varepsilon \varphi(\bar x)$ and hence, $\partial_\varepsilon \varphi(\bar x)$ is nonempty.
		
		Next we will show the closedness of $\partial_\varepsilon \varphi(\bar x)$. Assume that $\xi_k\in \partial_\varepsilon \varphi(\bar x)$ such that $\xi_k\to \xi$ as $k\to\infty$. By definition, we have
		\begin{equation*} 
			\langle\xi_k, x-\bar x\rangle-\varepsilon\leq \varphi(x)-\varphi(\bar x)\ \ \text{for all}\ \  x\in \R^n \ \ \text{and}\ \ k\in \N.
		\end{equation*} 
		Taking the limit as $k\to\infty$, the above inequality leads to
		\begin{equation*} 
			\langle\xi, x-\bar x\rangle-\varepsilon\leq \varphi(x)-\varphi(\bar x),\ \ \forall  x\in \R^n
		\end{equation*} 
		and so $\xi\in \partial_\varepsilon \varphi(\bar x)$, as required.
		
		To prove the convexity of $\partial_\varepsilon \varphi(\bar x)$, let $\xi_1, \xi_2\in \partial_\varepsilon \varphi(\bar x)$ and $\lambda\in (0,1)$. Then, for each $i=1, 2$, we have
		\begin{equation*}
			\langle\xi_i, x-\bar x\rangle-\varepsilon\leq \varphi(x)-\varphi(\bar x),\ \ \forall  x\in \R^n.
		\end{equation*}
		Hence,
		\begin{align*}
			\langle\lambda\xi_1+(1-\lambda)\xi_2, x-\bar x\rangle-\varepsilon&=\lambda[\langle\xi_1, x-\bar x\rangle-\varepsilon]+(1-\lambda)[\langle\xi_2, x-\bar x\rangle-\varepsilon]
			\\
			&\leq \lambda(\varphi(x)-\varphi(\bar x))+(1-\lambda)(\varphi(x)-\varphi(\bar x))
			\\
			&=\varphi(x)-\varphi(\bar x)
		\end{align*}
		for all $x\in\R^n$. This means that $\lambda\xi_1+(1-\lambda)\xi_2\in \partial_\varepsilon\varphi(\bar x)$ and the convexity of $\partial_\varepsilon \varphi(\bar x)$ follows.
		
		Lastly, we prove \eqref{equa_0}. By definition, it is clear that $\partial \varphi(\bar x)\subset \partial_\varepsilon \varphi(\bar x)$ for every $\varepsilon>0$. Hence,
		\begin{equation*} 
			\partial \varphi(\bar x)\subset\bigcap_{\varepsilon>0} \partial_\varepsilon \varphi(\bar x).
		\end{equation*}
		Conversely, assume that $\xi\in \bigcap\limits_{\varepsilon>0} \partial_\varepsilon \varphi(\bar x)$. Then, one has
		\begin{equation*}
			\langle\xi, x-\bar x\rangle-\varepsilon\leq \varphi(x)-\varphi(\bar x)\ \ \text{for all}\ \  x\in \R^n \ \ \text{and}\ \ \varepsilon>0.
		\end{equation*}
		Taking the limit as $\varepsilon\to 0$, the above inequality leads to 
		\begin{equation*}
			\langle\xi, x-\bar x\rangle\leq \varphi(x)-\varphi(\bar x) \ \ \forall x\in\R^n.
		\end{equation*}
		Thus $\xi\in \partial \varphi(\bar x)$ and so 
		\begin{equation*} 
			\bigcap_{\varepsilon>0} \partial_\varepsilon \varphi(\bar x)\subset\partial \varphi(\bar x),
		\end{equation*}
		which completes the proof.
	\end{proof}
	
	\begin{proposition}\label{pro2.3}
		Let $\varphi \colon\R^n\to \overline{\R}$ be a proper nearly convex function and $\varepsilon\geq 0$. Then 
		\begin{equation*}
			\partial_\varepsilon(\lambda \varphi)(\bar x)=\lambda\partial_{\varepsilon/\lambda}\varphi(\bar x) \ \ \forall\lambda>0, \forall \bar x\in\dom \varphi.
		\end{equation*}
	\end{proposition}
	
	\begin{proof} For $\lambda>0$, by the definition of the $\varepsilon$-subdifferential one has 
		\begin{align*}
			\xi\in   \partial_\varepsilon (\lambda \varphi)(\bar x) \ &\Leftrightarrow \langle \xi, x-\bar x \rangle -\varepsilon \le \lambda \varphi(x)-\lambda\varphi(\bar x), \ \forall x\in \dom \varphi \\  &\Leftrightarrow \tfrac{1}{\lambda}\langle \xi, x-\bar x\rangle -\tfrac{\varepsilon}{\lambda} \le \varphi(x)-\varphi(\bar x), \ \forall x\in \dom \varphi\\ 
			&\Leftrightarrow \tfrac{\xi}{\lambda} \in \partial_{\varepsilon/\lambda} \varphi (\bar x),\\
			&\Leftrightarrow \xi \in \lambda \partial_{\varepsilon/\lambda}\varphi(\bar x).
		\end{align*}
  The proof is complete.
	\end{proof}
	
	In convex analysis and optimization, summing two functions plays an important role. The Moreau–Rockafellar Theorem can be viewed as a well-known result, which describes the subdifferential of the sum of two subdifferentiable functions. In the same way, one can get a sum rule for $\varepsilon$-subdifferentials of two nearly convex functions.
	
	Firstly, we formally introduce the notion of conjugate below.
	\begin{definition}\rm 
		Consider a function $\varphi: \mathbb{R}^n \to \overline{\mathbb{R}}$. The \textit{conjugate} of $\varphi$ is
		$\varphi^*: \mathbb{R}^n \to \overline{\mathbb{R}}$ and is defined as
		$$\varphi^*(\xi)= \sup\limits_{x\in \mathbb{R}^n}\{\langle\xi,x\rangle-\varphi(x)\}.$$
	\end{definition}
	In order to study optimality conditions as well as sensitivity analysis of optimization problems, we need the subdifferential calculus rules, especially, sum rules play important roles among them. 
	The following proof for sum rules of nearly convex functions is based on the proof of~\cite[Theorem 3.1.1]{H_L_1993} and~\cite[Theorem 2.115]{Dhara-Dutta-12}.
	\begin{theorem}\label{sum-rule}
		Let $\varphi_i\colon\R^n\to\overline{\R}$, $i=1, 2$, be two proper nearly convex functions and $\varepsilon\ge 0$.   If the qualification condition
		\begin{equation}\label{QC_sum_rule}
			\ri(\dom\varphi_1)\cap\ri(\dom\varphi_2)\neq \emptyset
		\end{equation}
		holds, then  $\varphi_1+\varphi_2$ is nearly convex and
		\begin{equation*}
			\partial_\varepsilon (\varphi_1+\varphi_2)(\bar x)=\bigcup_{\substack{\varepsilon_1\geq 0, \varepsilon_2\geq 0, \\ \varepsilon_1+\varepsilon_2=\varepsilon}}[\partial_{\varepsilon_1}\varphi_1 (\bar x)+\partial_{\varepsilon_2}\varphi_2 (\bar x)]
		\end{equation*} 
		for all $\bar x\in \dom\varphi_1\cap\dom\varphi_2$.
	\end{theorem}
	\begin{proof}  The near convexity  of $\varphi_1+\varphi_2$ follows from \cite[Corollary 4.3]{Nam_Thieu_Yen_23}. Now, let    $\varepsilon_1 \ge 0$ and $\varepsilon_2 \ge 0$  be such that $\varepsilon_1+\varepsilon_2=\varepsilon$. The inclusion
		$$ \bigcup_{\substack{\varepsilon_1\geq 0, \varepsilon_2\geq 0, \\ \varepsilon_1+\varepsilon_2=\varepsilon}}[\partial_{\varepsilon_1}\varphi_1 (\bar x)+\partial_{\varepsilon_2}\varphi_2 (\bar x)] \subset \partial_\varepsilon (\varphi_1+\varphi_2)(\bar x)$$
		is based on the concept of $\varepsilon$-subdifferential. We now prove the opposite of this inclusion. Suppose that $\xi\in \partial_\varepsilon (\varphi_1+\varphi_2)(\bar x)$. By the definition of
		the $\varepsilon$-subdifferential, one has
		$$(\varphi_1+\varphi_2)(x)- (\varphi_1+\varphi_2)(\bar x) \ge \langle \xi, x-\bar x \rangle -\varepsilon, \ \forall x\in \mathbb{R}^n.$$
		From the definition of the conjugate function, $\xi\in \partial_\varepsilon (\varphi_1+\varphi_2)(\bar x)$ means that
		\begin{align}\label{ct1_new} (\varphi_1+\varphi_2)^*(\xi) +(\varphi_1+\varphi_2)(\bar x) -\langle \xi, \bar x\rangle \le \varepsilon.
		\end{align}
		Since the condition~\eqref{QC_sum_rule} holds, by~\cite[Theorem 6.6]{Huy_Nam_Yen_23} we have
			$$(\varphi_1+\varphi_2)^*(\xi)=(\varphi^*_1 \square \varphi^*_2)(\xi):=\inf\left\{\varphi^*_1(\xi_1) + \varphi^*_2(\xi_2) \mid \xi_1 + \xi_2 = \xi \right\}, \ \forall \xi \in \mathbb{R}^n$$
		and the infimum is attained. In other words, one can find $\xi_i\in \mathbb{R}^n, i=1,2$ with $\xi_1+\xi_2 =\xi$ such that $$(\varphi_1+\varphi_2)^*(\xi) =\varphi^*_1(\xi_1) +\varphi^*_2(\xi_2).$$
		Combining the latter with~\eqref{ct1_new} yields that
		$$(\varphi^*_1(\xi_1)+\varphi_1(\bar x) -\langle \xi_1, \bar x \rangle ) + (\varphi^*_2(\xi_2)+\varphi_2(\bar x) -\langle \xi_2, \bar x \rangle ) \le \varepsilon.$$
		Then, by defining $\varepsilon_i=\varphi^*_i(\xi_i)+\varphi_i(\bar x)- \langle \xi_i, \bar x \rangle, \ i=1,2$, one has $\varepsilon_i \ge 0$ and $\varepsilon_1+\varepsilon_2 \le \varepsilon$ for $i=1,2.$
		Moreover, by the definition of the conjugate function,
		\begin{align*}
			\varphi_i(x)-\varphi_i(\bar x) &\ge \langle \xi_i, x-\bar x \rangle-\varepsilon_i
			\\
			&\ge \langle \xi_i, x-\bar x \rangle-\bar\varepsilon_i,
		\end{align*}
		where $\bar \varepsilon_i=\varepsilon_i + \tfrac{\varepsilon-\varepsilon_1 -\varepsilon_2}{2}\ge \varepsilon_i, i=1,2$. So, for $i=1,2$,
		$\xi_i\in \partial_{\bar \varepsilon_i}\varphi_i(\bar x),$ 
		with $\bar\varepsilon_1+\bar \varepsilon_2=\varepsilon$. Therefore,
		$\xi=\xi_1+\xi_2\in \partial_{\bar \varepsilon_1} \varphi_1(\bar x) + \partial_{\bar \varepsilon_2} \varphi_2(\bar x) $.
		Since $\xi \in \partial_\varepsilon (\varphi_1+\varphi_2)(\bar x)$ was taken arbitrarily, we have
		$$\partial_\varepsilon (\varphi_1+\varphi_2)(\bar x) \subset \bigcup_{\substack{\varepsilon_1\geq 0, \varepsilon_2\geq 0, \\ \varepsilon_1+\varepsilon_2=\varepsilon}}[\partial_{\varepsilon_1}\varphi_1 (\bar x)+\partial_{\varepsilon_2}\varphi_2 (\bar x)],$$
		which completes the proof of the proposition.
	\end{proof}
	We consider the following illustrative example.
	\begin{example}\rm 
		Consider the nearly convex functions $\varphi_i: \R \to \overline{\mathbb{R}}$, $i=1,2,$ given, respectively,~by
		$$ \varphi_1(x)= \begin{cases} -\sqrt{x} & \mbox{if }\ \ x\in [0, 1),
			\\
			1& \mbox{if}\ \ x=1,
			\\
			\infty & \mbox{otherwise,}
		\end{cases}
		$$
		and
		$$ \varphi_2(x)= \begin{cases} -\sqrt{x} & \mbox{if } x\in [0, 1],
			\\
			\infty & \mbox{ otherwise.}
		\end{cases}
		$$  
		Then, one has
		$$ (\varphi_1+\varphi_2)(x)= \begin{cases} -2\sqrt{x} & \mbox{if } x\in [0, 1),
			\\
			0& \mbox{ if } x=1,\\
			\infty & \mbox{ otherwise}.
		\end{cases}
		$$
		Let $\bar x=0$. Obviously, $\bar x\in \dom \varphi_1 \cap \dom \varphi_2$ and the condition~\eqref{QC_sum_rule} holds. By similar way as in Example~\ref{ex1}, 
		for every $\varepsilon_1 > 0, \varepsilon_2>0$, one obtains
		\begin{align*}
			\partial_{\varepsilon_1} \varphi_1(\bar x)= \left(-\infty, -\tfrac{1}{4 \varepsilon_1} \right] \quad {\textrm{ and }} \quad \partial_{\varepsilon_2} \varphi_2(\bar x)= \left(-\infty, -\tfrac{1}{4 \varepsilon_2} \right].
		\end{align*}
  In the case, where $\varepsilon_1=0$ (resp., $\varepsilon_2=0$) one has $\partial_{\varepsilon_1}\varphi_1(\bar x)=\emptyset$ (resp., $\partial_{\varepsilon_2}\varphi_2(\bar x)=\emptyset$).
		Meanwhile, for $\varepsilon>0$, one has
		$$\partial_{\varepsilon} (\varphi_1+\varphi_2)(\bar x)= \left(-\infty, -\tfrac{1}{\varepsilon} \right].$$
		Consequently,  $$\partial_\varepsilon (\varphi_1+\varphi_2)(\bar x) = \bigcup_{\substack{\varepsilon_1\geq 0, \varepsilon_2\geq 0, \\ \varepsilon_1+\varepsilon_2=\varepsilon}}[\partial_{\varepsilon_1}\varphi_1 (\bar x)+\partial_{\varepsilon_2}\varphi_2 (\bar x)].$$  
	\end{example}
	
	\medskip
	
	When $\varepsilon=0$, one has the following corollary.
	\begin{corollary}
		Let $\varphi_i\colon\R^n\to\overline{\R}$, $i=1, 2$, be two proper nearly convex functions.   If~\eqref{QC_sum_rule} holds, then
		one has
		\begin{equation}\label{exact_sum_rule}
			\partial(\varphi_1+\varphi_2)(\bar x)=\partial\varphi_1 (\bar x)+\partial\varphi_2 (\bar x)
		\end{equation} 
		for all $\bar x\in \dom\varphi_1\cap\dom\varphi_2$.     
	\end{corollary}
	
	Let us consider an example that shows that~\eqref{exact_sum_rule} can be invalid providing that~\eqref{QC_sum_rule} is violated.
	\begin{example}\rm 
			Consider the following two functions:
		$$ \varphi_1(x)=\begin{cases}
			-\sqrt{x} & \mbox{if } x \in [0,1],\\
			\infty & \mbox{otherwise},
		\end{cases}$$
		and
		$$ \varphi_2(x)=\begin{cases}
			-\sqrt{-x} & \mbox{if } x \in [-1,0],\\
			\infty & \mbox{otherwise}.
		\end{cases}$$
		Clearly, $\varphi_1$ and $\varphi_2$ are convex functions. In particular, $\varphi_1$ and $\varphi_2$ are nearly convex functions. 
		By a simple calculation, one gets
		$$ (\varphi_1+\varphi_2)(x)=\begin{cases}
			0 & \mbox{if } x =0,\\
			\infty & \mbox{otherwise}.
		\end{cases}$$
		Let $\bar x = 0.$ On one hand, it is not difficult to see that 
		$\partial (\varphi_1+\varphi_2)(\bar x)=\mathbb{R}.$ On the other hand, one has $\partial\varphi_1 (\bar x)=\emptyset.$ Consequently,~\eqref{exact_sum_rule} does not hold. It is emphasized here that $\ri (\dom \varphi_1) \cap \ri (\dom\varphi_2) =\emptyset.$
	\end{example}
	
	\medskip
	
	From Theorem~\ref{sum-rule}, and the relationship between the $\varepsilon$-normal set and the $\varepsilon$-subdifferential  of the indicator function, the following corollary is obtained.
	\begin{corollary}\label{normal_rule} Let $\Omega_1,$ $\Omega_2$ be two nearly convex sets and $\varepsilon\geq 0$. If $\ri\Omega_1\cap\ri\Omega_2\neq \emptyset,$ then for each $\bar x\in \Omega_1\cap\Omega_2,$ we have
		\begin{equation*}
			N_\varepsilon(\bar x; \Omega_1\cap\Omega_2)=\bigcup_{\substack{\varepsilon_1\geq 0, \varepsilon_2\geq 0, \\ \varepsilon_1+\varepsilon_2=\varepsilon}}[N_{\varepsilon_1} (\bar x; \Omega_1)+N_{\varepsilon_2} (\bar x; \Omega_2)].
		\end{equation*}	 
		
	\end{corollary}
	The next corollary plays an important role in deriving the subsequent results, so one presents here a detailed proof.
	\begin{corollary}[Epsilon coderivative of sum]
		Suppose that $F_1, F_2: \mathbb{R}^n \rightrightarrows \mathbb{R}^m$ are nearly convex set-valued mappings and $\varepsilon\ge 0$. If the qualification condition
		\begin{align}
			\label{QC}
			\ri (\dom F_1) \cap \ri (\dom F_2)  \not=\emptyset
		\end{align}
		is satisfied, then one has
		\begin{equation}\label{coder-sum-rule}
			D^*_\varepsilon(F_1+F_2)(\bar x, \bar y)(v)=\bigcup_{\substack{\varepsilon_1\geq 0, \varepsilon_2\geq 0, \\ \varepsilon_1+\varepsilon_2=\varepsilon}}[D^*_{\varepsilon_1} F_1(\bar x, \bar y_1)(v)+D^*_{\varepsilon_2}F_2 (\bar x, \bar y_2)(v)] 
		\end{equation}
		holds for every $v\in \mathbb{R}^m$ and $(\bar y_1, \bar y_2) \in S(\bar x, \bar y)$, where
		$$S(\bar x, \bar y)=\{(\bar y_1, \bar y_2) \in \mathbb{R}^m \times \mathbb{R}^m \mid \bar y=\bar y_1 +\bar y_2, \ \bar y_i\in F_i(\bar x), i=1,2 \}.$$
	\end{corollary}
	\begin{proof}
		We first note that, the qualification condition~\eqref{QC} implies that $F_1+F_2$ is nearly convex by~\cite[Theorem 4.2]{Nam_Thieu_Yen_23}. For any $\varepsilon\ge 0$, $(\bar y_1, \bar y_2) \in S(\bar x, \bar y)$ and $v\in \mathbb{R}^m$, we take any $u\in  D^*_\varepsilon(F_1+F_2)(\bar x, \bar y)(v).$ By definition, one has 
		\begin{align}\label{ct2}
			(u, -v)\in N_\varepsilon ((\bar x, \bar y); \gph (F_1+F_2)). 
		\end{align}
		Consider two sets $\Omega_1$ and $\Omega_2$ defined as follows:
		\begin{align*}
			\Omega_1&=\{(x,y_1,y_2) \in \mathbb{R}^n\times \mathbb{R}^m \times \mathbb{R}^m \mid y_1 \in F_1(x) \}=\gph F_1 \times \mathbb{R}^m,\\
			\Omega_2&=\{(x,y_1,y_2) \in \mathbb{R}^n\times \mathbb{R}^m \times \mathbb{R}^m \mid y_2 \in F_2(x)\}.
		\end{align*}
		By the construction of $\Omega_1$ and $\Omega_2$, from~\eqref{ct2}, one has
		$$ (u,-v,-v) \in N_\varepsilon ((\bar x, \bar y_1, \bar y_2); \Omega_1 \cap \Omega_2). $$
		The qualification condition~\eqref{QC} implies that $\ri \Omega_1 \cap \ri \Omega_2 \not=\emptyset$ (see the proof of Theorem 4.2 in~\cite{Nam_Thieu_Yen_23}). So, one can apply Corollary~\ref{normal_rule} to obtain 
		$$(u, -v, -v) \in N_\varepsilon ((\bar x, \bar y_1, \bar y_2); \Omega_1 \cap \Omega_2) =\bigcup_{\substack{\varepsilon_1\geq 0, \varepsilon_2\geq 0, \\ \varepsilon_1+\varepsilon_2=\varepsilon}}[N_{\varepsilon_1} ((\bar x, \bar y_1, \bar y_2); \Omega_1) + N_{\varepsilon_2} ((\bar x, \bar y_1, \bar y_2); \Omega_2)]. $$
		Hence, there exist $\varepsilon_1\geq 0$ and  $\varepsilon_2\geq 0$ with $\varepsilon_1+\varepsilon_2=\varepsilon$ such that
		$$(u, -v, -v)\in N_{\varepsilon_1} ((\bar x, \bar y_1, \bar y_2); \Omega_1) + N_{\varepsilon_2} ((\bar x, \bar y_1, \bar y_2); \Omega_2).$$
		Consequently, one gets
		$$(u, -v, -v)=(u_1, v_1,w_1)+(u_2, v_2, w_2),$$
		where $(u_i, v_i, w_i)\in N_{\varepsilon_i}((\bar x, \bar y_1, \bar y_2); \Omega_i)$, $i=1, 2$. It is easy to see that $w_1=0$, $v_2=0$,  $(u_1, v_1)\in N_{\varepsilon_1}((\bar x, \bar y_1); \gph F_1)$, and $(u_2, w_2)\in N_{\varepsilon_2}((\bar x, \bar y_2); \gph F_2)$. So, $v_1=w_2=-v$ and $u=u_1+u_2$. 
		The latter means that 
		$$u=u_1+u_2 \in D^*_{\varepsilon_1} F_1(\bar x, \bar y_1)(v)+D^*_{\varepsilon_2}F_2 (\bar x; \bar y_2)(v).$$
		Hence,
		$$D^*_\varepsilon(F_1+F_2)(\bar x, \bar y)(v)\subset\bigcup_{\substack{\varepsilon_1\geq 0, \varepsilon_2\geq 0, \\ \varepsilon_1+\varepsilon_2=\varepsilon}}[D^*_{\varepsilon_1} F_1(\bar x, \bar y_1)(v)+D^*_{\varepsilon_2}F_2 (\bar x; \bar y_2)(v)]. $$
		The reverse of the above inclusion is obvious and we therefore get the sum rule \eqref{coder-sum-rule}. 
	\end{proof}
	We now discuss the $\varepsilon$-coderivative of the intersection mapping.
	\begin{theorem}
		Let $F_i: \mathbb{R}^n \rightrightarrows \mathbb{R}^m,$ for $i=1,2,\ldots,p$ be nearly convex set-valued mappings and let $F=\bigcap_{i=1}^p F_i$. Assume that the following qualification condition
		\begin{align}
			\label{CQ_set_valued} \bigcap\limits_{i=1}^p \ri (\gph F_i)\not=\emptyset
		\end{align}
		holds. Then, for any $\varepsilon \ge 0$ and $  (\bar x, \bar y)\in \gph F$ one has
		\begin{align}
			\label{inter_formula}
			D^*_\varepsilon F (\bar x, \bar y) (v)\!\!=\!\! \bigcup_{\substack{\varepsilon_1\geq 0,\ldots \varepsilon_p\geq 0, \\ \varepsilon_1+\ldots+\varepsilon_p=\varepsilon}} \!\{D^*_{\varepsilon_1} F_1 (\bar x, \bar y)(v_1)\!+...+\! D^*_{\varepsilon_p} F_p (\bar x, \bar y)(v_p) \mid\! v_1+...+v_p \!=\!v\}.
		\end{align}
	\end{theorem}
	\begin{proof}
		By the qualification condition~\eqref{CQ_set_valued}, the set-valued mapping $F=\bigcap_{i=1}^p F_i$ is nearly convex following~\cite[Theorem 4.9]{Nam_Thieu_Yen_23}. For $(\bar x, \bar y)\in \gph F$, take any $u\in D^*_\varepsilon F (\bar x, \bar y) (v)$. Then $(u,-v)\in N_\varepsilon ((\bar x, \bar y);\gph F)$. Since $\gph F =\bigcap\limits_{i=1}^p \gph F_i$, by~\eqref{CQ_set_valued} and Corollary~\ref{normal_rule} one has
		\begin{align}\label{ct17}
			(u,-v)\!\in\! N_\varepsilon ((\bar x, \bar y); \gph F)\!\!=\!\!\bigcup_{\substack{\varepsilon_1\geq 0,\ldots, \varepsilon_p\geq 0, \\ \varepsilon_1+\ldots+\varepsilon_p=\varepsilon}} \!\big[ N_{\varepsilon_1}((\bar x, \bar y); \gph F_1)\!+\!\ldots\!+\!N_{\varepsilon_p}((\bar x, \bar y); \gph F_p) \big].
		\end{align}
		Then, we can find $u_1,\ldots,u_p \in \mathbb{R}^n $ and $v_1,\ldots,v_p\in \mathbb{R}^m$ such that
		$$ u=\sum\limits_{i=1}^p u_i, \ v=\sum\limits_{i=1}^p v_i, \ (u_i,-v_i)\in N_{\varepsilon_i} ((\bar x, \bar y);\gph F_i).$$
		By the definition, $u_i\in D^*_{\varepsilon_i} F_i(\bar x, \bar y)(v_i)$ and hence
		$$u\in \bigcup_{\substack{\varepsilon_1\geq 0,\ldots, \varepsilon_p\geq 0, \\ \varepsilon_1+\ldots+\varepsilon_p=\varepsilon}}  \big[ N_{\varepsilon_1}((\bar x, \bar y); \gph F_1)\!+\!\ldots\!+\!N_{\varepsilon_p}((\bar x, \bar y); \gph F_p) \big] ,$$
		which justifies the inclusion $``\subseteq"$ in~\eqref{inter_formula}. The reverse inclusion follows directly from the definition of the $\varepsilon$-coderivative and the equality in~\eqref{ct17}.
	\end{proof}
	
	\section{Optimality conditions}
	In this section, we will study the optimality conditions for nearly convex optimization problems by using the $\varepsilon$-subdifferential.
	Consider the following optimization problem
	\begin{equation}\label{problem}
		\inf\, \varphi(x) \ \ \text{subject to}\ \ x\in S,\tag{P}	
	\end{equation} 
	where $S\subset \R^n$ is nonempty and $\varphi\colon\R^n\to\overline{\R}$ is bounded from below on $S$. When $S=\mathbb{R}^n$, we call~\eqref{problem} the unconstrained optimization problem.
	\begin{definition}
		\rm 
		Let $\varepsilon>0$ and $\bar x\in S$. A vector $\bar x$ is said to be an \textit{$\varepsilon$-solution }of \eqref{problem} if 
		\begin{equation*}
			\varphi(\bar x)\leq \varphi(x)+\varepsilon \ \ \forall x\in S.
		\end{equation*}
	\end{definition}
	\begin{remark}\rm 
		From the definition of the $\varepsilon$-subdifferential, it is easy to see that for $\varepsilon\ge 0$, $\bar x$ is an $\varepsilon$-solution of the unconstrained optimization problem if and only if $0\in \partial_\varepsilon \varphi(\bar x).$  
	\end{remark} 
	
	We are now in a position to prove the necessary as well as the sufficient approximate optimality conditions of the nearly optimization problem~\eqref{problem}.
	\begin{theorem}\label{Theorem-1} Assume that $\varphi$ is a nearly convex function, $S$ is a nearly convex set, and the following qualification condition holds
		\begin{equation}\label{RC_1}
			\ri(\dom \varphi)\cap\ri S\neq \emptyset.
		\end{equation} 
		Then, $\bar x\in S$ is an $\varepsilon$-solution of \eqref{problem} if and only if there exist $\varepsilon_1, \varepsilon_2\geq 0$ with $\varepsilon_1+\varepsilon_2=\varepsilon$ such~that
		\begin{equation}\label{ct20}
			0\in \partial_{\varepsilon_1}\varphi(\bar x)+ N_{\varepsilon_2}(\bar x; S).
		\end{equation} 
	\end{theorem}
	\begin{proof}
		Using the fact that $\bar x$ is an $\varepsilon$-solution of \eqref{problem} if and only if $0\in \partial_\varepsilon(\varphi+\delta_S(\cdot))(\bar x)$. In addition, we note that $\dom \delta_S = S$.
		Now, applying Theorem \ref{sum-rule}, where $\varphi$ and $\delta_S$ play the corresponding roles of $\varphi_1$ and $\varphi_2$, one has
		\begin{align*}
			\partial_\varepsilon(\varphi+\delta_S(\cdot))(\bar x) &= \bigcup_{\substack{\varepsilon_1\geq 0, \varepsilon_2\geq 0, \\ \varepsilon_1+\varepsilon_2=\varepsilon}}[\partial_{\varepsilon_1}\varphi (\bar x)+\partial_{\varepsilon_2}\delta_S (\bar x)]\\
			&=\bigcup_{\substack{\varepsilon_1\geq 0, \varepsilon_2\geq 0, \\ \varepsilon_1+\varepsilon_2=\varepsilon}}[\partial_{\varepsilon_1}\varphi (\bar x)+N_{\varepsilon_2}(\bar x;S)],
		\end{align*}
		which justifies~\eqref{ct20}.
	\end{proof}
	We now give an illustration for Theorem~\ref{Theorem-1}.
	\begin{example}\rm 
		Consider the problem~\eqref{problem} where $\varphi: \mathbb{R}\to \overline{\R}$ and $S$ are given, respectively,~by
		$$ \varphi(x)=\begin{cases}
			|x| & \mbox{ if } x\in [-1,1),\\
			2 & \mbox{ if } x=1,\\
			\infty &\mbox{ otherwise},
		\end{cases}$$
		and $S=[0, \infty).$ It is clear that $\varphi$ is a nearly convex function, $S$ is a convex set, and~\eqref{RC_1} is satisfied. In this case, we observe that $\bar x=0$ is an $\varepsilon$-solution of~\eqref{problem} for every $\varepsilon \ge 0.$ Meanwhile, for $\varepsilon_1 \ge 0$, $\varepsilon_2 \ge 0$,
		by the example in~\cite[pp. 93–94]{H_L_1993} one has $\partial_{\varepsilon_1} \varphi (\bar x)=[-1,1]$. Additionally, from~\cite[Example 2.1]{Hiriart_1989} $N_{\varepsilon_2} (\bar x; S)= N (\bar x; S)=(-\infty, 0].$ Consequently,
		$$0\in \partial_{\varepsilon_1} \varphi(\bar x) +  N_{\varepsilon_2} (\bar x; S).$$  
	\end{example}
	
	\medskip
	Given a nearly convex function  $f: \R^n\times \R^p\to \overline{\R}$ and a nearly set-valued mapping $G: \R^n \rightrightarrows \R^p$, we consider the \textit{parametric nearly convex optimization problem under an inclusion constraint}
	\begin{align}\label{Px}\tag{$P_x$}
		\quad \quad \quad  	\min\{f(x,y)\mid y \in G(x)\} 
	\end{align}
	depending on the parameter $x$.
	
	The next theorem describes the necessary and sufficient approximate optimality conditions for $(P_x)$ at a given parameter $\bar x \in \R^n$. 
	
	\begin{theorem}Let $\bar x \in \R^n$. Suppose that $f(\bar x, \cdot)$ is bounded from below on $G(\bar x)$ and the following qualification condition
		\begin{align*}
			\ri (\dom f(\bar x, \cdot) ) \cap \ri ( G(\bar x))\not= \emptyset
		\end{align*}
		is fulfilled.
		Then, $\bar y\in G(\bar {x})$ is an  $\varepsilon$-solution of \eqref{Px} if and only if there exist $\varepsilon_1, \varepsilon_2 \ge 0$ with $\varepsilon_1+\varepsilon_2=\varepsilon$ such that
		\begin{align} \label{NS_condition}
			0 \in \partial^y_{\varepsilon_1} f(\bar x, \bar y) + N_{\varepsilon_2}(\bar y; G(\bar x)),
		\end{align}
		where $\partial^y_{\varepsilon_1} f(\bar x, \bar y) $ is the $\varepsilon_1$-subdifferential of the function $f(\bar x, \cdot)$ with respect to the variable $y$.
	\end{theorem}
	
	\begin{proof}
		Let $\bar x \in \mathbb{R}^n$.		Consider the function $\varphi_G(y)= f(\bar x, y) + \delta_{G(\bar x)}( y)$, where $\delta_{G(\bar x)}(\cdot)$ is the indicator function of the nearly convex set $G(\bar x)$. The latter means that $\delta_{G(\bar x)}(y)=0$ for $y\in G(\bar x)$ and $\delta_{G(\bar x)}(y)=\infty$ for $y\notin G(\bar x)$. We now apply Theorem~\ref{Theorem-1} for the case where $ f(\bar x, \cdot)$ and $G(\bar x)$ play the corresponding the role of $\varphi$ and $S$. Consequently, we obtain~\eqref{NS_condition}.
	\end{proof}


	\section{Sensitivity Analysis}
	We now present some new results on sensitivity analysis of nearly convex optimization problems under inclusion constraints. By using the sum rules and appropriate qualification conditions, we will obtain formulas for computing the $\varepsilon$-subdifferential of the optimal value functions in both parametric unconstrained and constrained nearly convex optimization problems.
	\subsection{Unconstrained nearly convex optimization problem}
	
	Consider the \textit{parametric unconstrained nearly convex optimization problem}
	\begin{align}\label{math_pro.}
		\min\{f(x,y)\mid  y \in \mathbb{R}^p\}
	\end{align}
	depending on the parameter $x$, where $f: \mathbb{R}^n \times \mathbb{R}^p  \to \overline{\mathbb R}$. The function $f$ is called the \textit{objective function} of~\eqref{math_pro.}. The \textit{optimal value function}
	$m: \mathbb{R}^n \rightarrow \overline{\R}$ of \eqref{math_pro.} is
	\begin{align}\label{marg.func.}
		m(x):= \inf \left\{f (x,y)\mid  y \in \mathbb{R}^p\right\}.
	\end{align}
	The \textit{solution set} of \eqref{math_pro.} with respect to a given parameter $\bar x\in\R^n$ is defined by
$$S(\bar x):=\{y \in \mathbb{R}^p\mid m(\bar x)= f (\bar x,y)\}.$$
For $\eta > 0$, one calls $S_\eta(\bar x):= \{ y \in \mathbb{R}^p \mid  f (\bar x, y) \le m (\bar x)+\eta \}$
	the \textit{approximate solution set} of~\eqref{math_pro.}.
	
	\medskip
	
	Under the near convexity of the objective function $f$, one has the optimal value function $m$ is nearly convex. This result is a direct consequence of Theorem 5.2 in \cite{Huy_Nam_Yen_23}.
	\begin{proposition}
		Let $f: \mathbb{R}^n \times \mathbb{R}^p  \to \overline{\mathbb R}$ be a proper nearly convex function. Then, the optimal value function $m$ given in~\eqref{marg.func.} is nearly convex.
	\end{proposition}

	We now obtain formulas for the $\varepsilon$-subdifferential of $m$. In the convex setting, the following result can be found somewhere, for example, it was given in  \cite[Corollary~5]{MoussaouiSeeger1994} as a consequence of a more general result. The interested readers also see this result in \cite[Theorem~2.6.2]{Zalinescu_2002} and in~\cite[Theorem 4.1]{An_Yao_19}. We now present a detailed, direct proof in the case of nearly convex case. Our arguments are based on a proof scheme of \cite{MoussaouiSeeger1994}.	
	
	\begin{theorem}\label{sensitivity1} Let $f:  \mathbb{R}^n \times \mathbb{R}^p   \to \overline{\R}$ be a proper nearly convex function and $m$ be finite at $\bar x \in \mathbb{R}^n.$ Then, for every $\varepsilon \ge 0$,
		\begin{equation}\label{Epsilon_subdifferential2}
			\begin{array}{rcl}\partial_\varepsilon m(\bar x) &=& \bigcap\limits_{\eta\; >\;0} \ \bigcap\limits_{y \,\in\, S_\eta (\bar x)} \bigg\{\xi \in \mathbb{R}^n \mid  (\xi,0) \in \partial_{\varepsilon +\eta} f (\bar x, y) \bigg\}\\
				&=&\bigcap\limits_{\eta\; >\;0} \ \bigcup\limits_{y \,\in\, \mathbb{R}^p} \bigg\{\xi \in \mathbb{R}^n \mid  (\xi,0) \in \partial_{\varepsilon +\eta} f (\bar x, y) \bigg\}. 
			\end{array}
		\end{equation}
		In particular, 
		\begin{equation}\label{subdifferential}
			\begin{array}{rcl}	\partial m(\bar x) &=& \bigcap\limits_{\eta\; >\;0} \ \bigcap\limits_{y \,\in\, S_\eta (\bar x)} \bigg\{\xi \in \mathbb{R}^n \mid (\xi,0) \in \partial_{\eta} f (\bar x, y) \bigg\}\\
				&=&\bigcap\limits_{\eta\; >\;0} \ \bigcup\limits_{y \,\in\, \mathbb{R}^p} \bigg\{\xi \in \mathbb{R}^n \mid  (\xi,0) \in \partial_{\eta} f (\bar x, y) \bigg\}.
			\end{array}
		\end{equation}
		In addition, if $S(\bar x)\not= \emptyset$, then for every $\varepsilon \ge 0$, one has
		\begin{align}\label{nonconstraint}
			\partial_\varepsilon m(\bar x) = \big\{\xi \in \mathbb{R}^n \mid  (\xi,0) \in \partial_{\varepsilon} f (\bar x, y) \big\},
		\end{align}
		for all $y \in S(\bar x).$
	\end{theorem}
	\begin{proof}
		Set
		\begin{align*}
			&\mathcal{A}_\eta (\bar x)= \bigcap\limits_{y \,\in\, S_\eta (\bar x)} \bigg\{\xi \in \mathbb{R}^n \mid (\xi,0) \in \partial _{\varepsilon +\eta} f (\bar x, y)
			\bigg\},\\
			&\mathcal{B}_\eta (\bar x)= \bigcup\limits_{y \,\in\, \mathbb{R}^p} \bigg\{\xi \in \mathbb{R}^p \mid (\xi,0) \in \partial _{\varepsilon +\eta} f (\bar x, y)
			\bigg\}.
		\end{align*}
		So~\eqref{Epsilon_subdifferential2} means 
		$$\partial_\varepsilon m(\bar x) =  \bigcap\limits_{\eta\,>\,0}\mathcal{A}_\eta (\bar x)= \bigcap\limits_{\eta\,>\,0} \mathcal{B}_\eta (\bar x).$$
		We first observe that the set $S_\eta(\bar x)$ is nonempty for every $\eta >0$ as $m(\bar x)=\inf\limits_{y\, \in\, \mathbb{R}^p} \varphi (\bar x, y)$ by~\eqref{marg.func.}. Thus, it is obvious $\mathcal{A}_\eta (\bar x)\subset \mathcal{B}_\eta (\bar x)$ for all $\eta >0$, and hence $\bigcap\limits_{\eta>0}\mathcal{A}_\eta (\bar x)\subset \bigcap\limits_{\eta>0} \mathcal{B}_\eta (\bar x)$. So, the equalities in \eqref{Epsilon_subdifferential2} will be proved, if we can clarify that
		\begin{equation}
			\label{ct1}
			\partial_\varepsilon m(\bar x) \subset \bigcap\limits_{\eta\,>\,0}\mathcal{A}_\eta (\bar x)\end{equation} and
		\begin{equation}\label{ct1n}
			\bigcap\limits_{\eta\,>\,0} \mathcal{B}_\eta (\bar x) \subset \partial_\varepsilon m(\bar x).
		\end{equation}
		To show \eqref{ct1}, we now take any $\xi \in \partial_\varepsilon m(\bar x)$, $\eta>0$, and $y \in S_\eta(\bar x)$. By the definition of the conjugate function, we observe that $\xi \in \partial_\varepsilon m(\bar x)$ if and only if
		\begin{align}
			\label{ct2n}
			m(\bar x) + m^*(\xi) \le \langle \xi, \bar x \rangle +\varepsilon.
		\end{align}
		Adding $\eta>0$ to both sides of \eqref{ct2n} gives
		\begin{align}
			\label{ct3}
			m(\bar x) + m^*(\xi) + \eta\le \langle \xi, \bar x \rangle+\varepsilon +\eta.
		\end{align}
		Since $y \in S_\eta (\bar x)$, one has $f(\bar x, y) \le m(\bar x)+\eta$. So, \eqref{ct3} yields
		\begin{align}
			\label{ct4}
			f(\bar x,y) + m^*(\xi) \le \langle \xi, \bar x \rangle +\varepsilon +\eta.
		\end{align}
		For every $\zeta \in \mathbb{R}^n$, by the definition of conjugate function, we have
		\begin{align*}
			m^*(\zeta)&= \sup\limits_{x \,\in\, \mathbb{R}^n} \big \{ \langle \zeta, x \rangle - m (x)\big \}\\ & = \sup\limits_{x \in \mathbb{R}^n} \big \{ \langle \zeta, x \rangle - \inf\limits_{y \,\in\, \mathbb{R}^p} f (x,y)\big \}\\
			&= \sup\limits_{(x,y) \,\in\, \mathbb{R}^n\times \mathbb{R}^p} \big \{ \langle \zeta, x \rangle - f (x,y)\big \}\\
			&= \sup\limits_{(x,y) \,\in \,\mathbb{R}^n\times \mathbb{R}^p} \big \{ \langle (\zeta,0), (x,y) \rangle - f (x,y)\big \}\\
			&= f^*(\zeta,0).
		\end{align*}
		Replacing $m^*(x^*)=f^*(\zeta,0)$	into \eqref{ct4}, one obtains
		\begin{align}
			\label{ct5}
			f(\bar x,y) + f^*(\xi,0) \le \langle \xi, \bar x \rangle +\varepsilon +\eta.
		\end{align}
		Consequently,  inequality~\eqref{ct5} yields $(\xi,0)\in \partial_{\varepsilon +\eta} f (\bar x, y)$ for all $\eta>0$ and $y \in S_\eta(\bar x).$ In the other words $\xi \in \bigcap\limits_{\eta\,>\,0}\mathcal{A}_\eta (\bar x)$, so \eqref{ct1} is valid.
		
		Next, to prove \eqref{ct1n}, take any $x^* \in \bigcap\limits_{\eta>0}\mathcal{B}_\eta (\bar x)$. Then, for every $\eta>0$, one can find $y \in \mathbb{R}^p$ such that $(\xi,0)\in \partial_{\varepsilon +\eta} \varphi(\bar x, y)$. The latter means that
		$$
		f^*(\xi,0) +f (\bar x,y) - \langle (\xi,0), (\bar x, y) \rangle \le \varepsilon +\eta,
		$$
		or, equivalently,
		\begin{align}\label{ct6}
			f^*(\xi,0) +f (\bar x,y) - \langle \xi,\bar x \rangle \le \varepsilon +\eta.
		\end{align}
		Since $ f^*(\xi,0)=m^*(\xi)$ and $ m(\bar x) \le \varphi (\bar x, y)$, \eqref{ct6} implies
		\begin{align}\label{ct7}
			m^*(\xi) +m (\bar x) - \langle \xi, \bar x \rangle \le \varepsilon +\eta.
		\end{align}  
		As \eqref{ct7} holds for every $\eta >0$, letting $\eta \to 0$ yields
		$$ m^*(\xi) +m (\bar x) - \langle \xi, \bar x \rangle \le \varepsilon.$$
		The last inequality gives $\xi  \in \partial_\varepsilon m(\bar x)$. Therefore, \eqref{ct1n} is satisfied.
		
		Combining \eqref{ct1} and \eqref{ct1n} yields \eqref{Epsilon_subdifferential2}. For $\varepsilon=0$, from \eqref{Epsilon_subdifferential2} one obtains~\eqref{subdifferential}.
	\end{proof}
	
	Here is a simple example designed to illustrate Theorem~\ref{sensitivity1}.
	\begin{example} \rm  Consider
 a function $f: \R^2 \to \overline{\R}$ given by
		\[ f(x,y) = \begin{cases} 
			x^2+y^2 & \text{if } (x,y)\in [-1,1]\times[-1,1]\setminus \left(\tfrac{1}{2},1\right),\\
			\infty & \text{otherwise}.\\
		\end{cases}
		\]
		
		Then the optimal value function~\eqref{marg.func.} of problem~\eqref{math_pro.} is
		$$ m(x):= \inf \left\{ f(x,y)\mid  y \in \Bbb {R}\right\} =x^2. $$	
		Let $\bar x=0$. For every $\varepsilon\ge 0$, one has 
		\begin{align*}
			\partial_{\varepsilon}m(\bar x)&=\left\{ x^*\in\Bbb {R} \mid x^*x \le x^2+\varepsilon, \forall x\in\Bbb{R}\right \}\\
			&=\left\{ x^*\in\Bbb {R} \mid  -x^2+x^*x-\varepsilon\le 0 , \forall x\in\Bbb{R} \right\}\\
			&=[-2\sqrt{\varepsilon},2\sqrt{\varepsilon}].
		\end{align*}
		In this case, $\bar y=0 \in S(\bar x)$. Thus we will show~\eqref{nonconstraint} is valid.  It is easy to see that,  for any function $\varphi\colon \R\times \R \to \R$ of the form $\varphi(x,y)=\varphi_1(x) + \varphi_2(y)$ for all $(x, y)\in \R\times \R$, one has
		$\partial_{\varepsilon} \varphi(x,y)\ \subseteq\ \partial_{\varepsilon}\varphi_1(x)\times \partial_{\varepsilon}\varphi_2(y)$.
		We now apply this for the function $f$ to obtain
		$\partial_{\varepsilon} f(\bar x,\bar y)\ {\subseteq}\		[-2\sqrt{\varepsilon},2\sqrt{\varepsilon}]\times [-2\sqrt{\varepsilon},2\sqrt{\varepsilon}]$  and so
		\begin{align*}
			RHS_{\eqref{nonconstraint}} &=\left\lbrace x^*\in \mathbb{R} \mid (x^*,0)\in \partial_{\varepsilon}f(\bar{x},y) \right\rbrace \\
			&=\{ x^*\in \mathbb{R} \mid (x^*,0)\in \partial_{\varepsilon}f(\bar{x},\bar y) \} \\
			&=\{ x^*\in \mathbb{R} \mid (x^*,0)\in [-2\sqrt{\varepsilon},2\sqrt{\varepsilon}]\times[-2\sqrt{\varepsilon},2\sqrt{\varepsilon}]\} \\
			&=[-2\sqrt{\varepsilon},2\sqrt{\varepsilon}].
		\end{align*}
		Therefore, the conclusion of Theorem~\ref{sensitivity1} is justified.
	\end{example}
	\subsection{Constrained nearly convex optimization problems}
	In this subsection, we will extend the result of the previous subsection for constrained nearly convex optimization problems by using some suitable qualification conditions.

	Let $f: \mathbb{R}^n \times \mathbb{R}^p \rightarrow \overline{\R}$ be an extended real-valued function, and $G: \mathbb{R}^n  \rightrightarrows \mathbb{R}^p$ be a set-valued mapping.	Consider the \textit{parametric optimization problem under an inclusion constraint}
	\begin{align}\label{math_program1}
		\min\{f(x,y)\mid y \in G(x)\}
	\end{align}
	depending on the parameter $x$. The \textit{optimal value function}
	$m: \mathbb{R}^n  \rightarrow \overline{\R}$ of \eqref{math_program1} is
	\begin{align}\label{marginalfunction1}
		m(x):= \inf \left\{f (x,y)\mid y \in G(x)\right\}.
	\end{align}
	The usual convention $\inf \emptyset =\infty$ forces $m(x)=\infty$ for every $x \notin {\rm{dom}}\, G.$ 
	The \textit{solution map}  $S: \mathbb{R}^n  \rightrightarrows \mathbb{R}^p $ of \eqref{math_program1} is defined by
	\begin{align*}
		S(x)=\{y \in G(x)\mid  m(x)= f (x,y)\}.
	\end{align*}
	The \textit{approximate solution set} of \eqref{math_program1} is given by 
	\begin{align}\label{solutionset}
		S_\eta(\bar x)= \{ y \in G(\bar x) \mid f (\bar x, y) \le m (\bar x)+\eta \},\ \, \forall \eta > 0.
	\end{align} 
	
	The following result on the near convexity of the optimal value function is proved in~\cite[Theorem 5.2]{Huy_Nam_Yen_23}.
	
	\begin{proposition}\label{Pro5.2}
		Suppose that $f: \mathbb{R}^n \times \mathbb{R}^p \to \overline{\R}$  is a proper nearly convex function, and $G: \mathbb{R}^n \rightrightarrows \mathbb{R}^p$
		is a nearly convex set-valued mapping. Suppose furthermore that
		\begin{align*}
			\ri (\dom f) \cap \ri (\gph G) \not= \emptyset.
		\end{align*}
		Then the optimal value function $m$ given in~\eqref{marginalfunction1} is nearly convex.
	\end{proposition}
	
	\medskip	
	
	For any $\varepsilon\ge 0$ and $ \eta \ge 0$, define by  $\varGamma(\eta+\varepsilon)$ the set $$\varGamma(\eta+\varepsilon)=\{(\gamma_1, \gamma_2) \mid \gamma_1 \ge 0,\ \gamma_2 \ge 0, \ \gamma_1+\gamma_2=\eta+\varepsilon\}.$$
	
	We are now in a position to formulate the main result of this subsection.
	\begin{theorem}
		\label{thm_sensitivity_1} Suppose that $f: \mathbb{R}^n \times \mathbb{R}^p \rightarrow \overline{\R}$ is a proper nearly convex function, and $G: \R^n \rightrightarrows \R^p$ is a nearly convex set-valued mapping. Assume that the optimal value function $m$ in \eqref{marginalfunction1} is finite at $\bar x \in \mathbb{R}^n.$ If the following qualification condition
		\begin{align}\label{RC}
			\ri(\dom f ) \cap \ri(\gph G ) \not=\emptyset
		\end{align}
		holds. Then, for every $\varepsilon \ge 0$, we have
		\begin{equation*}
			\begin{split}
				\partial _\varepsilon m(\bar x)&=
				\bigcap\limits_{\eta >0} 
				\bigcap\limits_{ y \,\in\, S_\eta (\bar x)}
				\bigcup_{(\gamma_1, \gamma_2)\, \in \, \varGamma(\eta+\varepsilon)}
				\bigg\{\xi\in \mathbb{R}^n \mid (\xi,0) \in \partial_{\gamma_1}f(\bar x,y) \!+\!N_{\gamma_2}\big ((\bar x,  y); {\rm gph}\, G \big ) \bigg\}\\
				&=\bigcap\limits_{\eta >0} 
				\ \bigcup\limits_{y\, \in \, \mathbb{R}^p}\
				\bigcup_{(\gamma_1, \gamma_2) \,\in\, \varGamma(\eta+\varepsilon)}
				\bigg\{\xi \in \mathbb{R}^n  \mid (\xi,0) \in \partial_{\gamma_1}f(\bar x, y) +N_{\gamma_2}\big ((\bar x, y); {\rm gph}\, G \big ) \bigg\},
			\end{split}
		\end{equation*}
		where $S_\eta (\bar x)$ is given in \eqref{solutionset}.
	\end{theorem}	
	\begin{proof}
		Since the qualification condition~\ref{RC} holds, it follows that the optimal value function $m$ is nearly convex by Proposition~\ref{Pro5.2}.
		We apply Theorem \ref{sensitivity1} to the case where $f (x,y)$ plays the role of $\big (f +\delta_{\gph G} (\cdot)\big )(x,y)$, where $\delta_{\gph G}(\cdot)$ is the indicator function of $\gph G$. Thus
		\begin{equation}\label{N_Formula_1}
			\begin{split}
				\partial_\varepsilon m (\bar x)& = \bigcap\limits_{\eta >0} 
				\ \bigcap\limits_{ y \,\in\, S_\eta (\bar x)}\bigg\{\xi \in \mathbb{R}^n \mid (\xi,0) \in \partial_{\varepsilon +\eta} \big (f +\delta_{\gph G} (\cdot)\big ) (\bar x,  y) \bigg\}\\
				&	=\bigcap\limits_{\eta >0} \ \bigcup\limits_{  y\, \in \, \mathbb{R}^p} \bigg\{\xi \in \mathbb{R}^n \mid (\xi,0) \in \partial_{\varepsilon +\eta} \big (f +\delta_{\gph G}(\cdot)\big ) (\bar x,  y) \bigg\}.
			\end{split}
		\end{equation}
		We will clarify that 
		\begin{align}\label{N_Formula_2}
			\partial_{\varepsilon +\eta}\!\big (f +\delta_{\gph G}(\cdot)\big ) (\bar x,  y)\!\!=\! \! \bigcup_{(\gamma_1, \gamma_2) \,\in\, \varGamma(\eta+ \varepsilon)}
			\bigg\{ \partial_{\gamma_1}f(\bar x, y)\! +\!N_{\gamma_2}((\bar x, y); {\rm gph}\, G)  \bigg\},
		\end{align}		
		where $ \varGamma(\eta+ \varepsilon)=\{(\gamma_1, \gamma_2) \mid \gamma_1 \ge 0,\ \gamma_2 \ge 0,\ \gamma_1+\gamma_2=\varepsilon+\eta\}.$  
		Indeed, suppose that~\eqref{RC} is fulfilled. Since ${\rm gph}\,G$ is nearly convex, $\delta_{\gph G}: \mathbb{R}^n \times \mathbb{R}^p \to \overline {\R}$ is nearly convex. Moreover, $\dom \delta_{\gph G}={\rm gph}\,G$. So thanks to Theorem~\ref{sum-rule}, we have
		\begin{displaymath}
			\partial_{\varepsilon+\eta} \big (f +\delta_{\gph G}(\cdot)\big )(\bar x, y)=
			\bigcup_{(\gamma_1, \gamma_2) \,\in\, \varGamma(\eta+ \varepsilon)} \bigg \{  \partial_{\gamma_1}f(\bar x,  y)+ \partial_{\gamma_2}\delta_{\gph G}((\bar x, y))\bigg \},
		\end{displaymath}
		for any $(\bar x, y) \,\in\, {\rm dom}\, f \cap\, {\rm gph}\, G$. Meanwhile, by the definition $$\partial_{\gamma_2}\delta_{\gph G}\big ((\bar x, y)\big )= N_{\gamma_2}\big ((\bar x, y); {\rm gph}\, G \big ).$$ Combining \eqref{N_Formula_1} with \eqref{N_Formula_2} gives the statement of the theorem.
	\end{proof}
	
	\medskip
	
	Let us consider an illustrative example for Theorem~\ref{thm_sensitivity_1}.
	\begin{example}\rm Let $G(x)=\{y \in \Bbb{R}\mid  y \ge |x| \}$ and
		$$ f(x,y)= \begin{cases} \tfrac{1}{2}|x|+\tfrac{3}{2}|y| & \mbox{if } (x,y)\in [-1,1] \times [-1,1] \setminus \left(\tfrac{1}{2},1\right),\\
			\infty & \mbox{ otherwise}.
		\end{cases}
		$$	
		Then the optimal value function $m$ of problem~\eqref{math_program1} is 		
		$$
		m(x):= \inf \left\{f (x,y)\mid  y \in G(x)\right\}=2|x|, \ \mbox{for all } x\in \mathbb{R}.
		$$
		For $\bar x = 0,$ the approximate solution set is
		\begin{align*}
			S_\eta(\bar x)&= \{ y \in G(\bar x) \mid f (\bar x, y) \le m (\bar x)+\eta \},\ \, \forall \eta > 0\\
			&=\{ y \in G(\bar x) \mid \tfrac{3}{2}|y| \le \eta \},\ \, \forall \eta > 0\\
			&=\{0\}.
		\end{align*} 
		For every $\varepsilon \ge 0$, from~\cite[pp. 93–94]{H_L_1993} and Proposition~\ref{pro2.3}, one has
		$$\partial_{\varepsilon} m(x)=\partial_{\varepsilon}(2|x|)=2\partial_{\frac{\varepsilon}{2}}(|x|)=\begin{cases} \left[-2,-2-\tfrac{\varepsilon}{x}\right] & \mbox{if } x<-\tfrac{1}{4}\varepsilon,\\
			[-2,2] & \mbox{if } -\tfrac{1}{4}\varepsilon \le x\le\tfrac{1}{4}\varepsilon,\\
			\left[2-\tfrac{2\varepsilon}{x},2\right] & \mbox{if } x>\tfrac{1}{4}\varepsilon.
		\end{cases}
		$$
		Hence,
		$\partial_{\varepsilon} m(\bar x)=[-2,2].$ For any $\gamma_1 \ge 0$, $\gamma_2\ge 0$, one has
		$\partial_{\gamma_1} f(x,y)\ \subseteq\ \partial_{\gamma_1}f_{1}(x)\times \partial_{\gamma_1}f_{2}(y)$
		where $f_1(x) = \tfrac{1}{2}|x|, f_2(y)=\tfrac{3}{2}|y|$. From Proposition~\ref{pro2.3} and~\cite[pp. 93–94]{H_L_1993}, one gets
		$$\partial_{\gamma_1} f_{1}(x)=\partial_{\gamma_1}\left(\tfrac{1}{2}|x|\right)=\tfrac{1}{2}\partial_{2\gamma_1}(|x|)=\begin{cases} \left[-\tfrac{1}{2},-\tfrac{1}{2}-\tfrac{\gamma_1}{x}\right] & \mbox{if } x<-\gamma_1,\\
			\\
			\left[-\tfrac{1}{2},\tfrac{1}{2}\right] & \mbox{if } -\gamma_1 \le x\le\gamma_1,\\
			\\
			\left[\tfrac{1}{2}-\tfrac{\gamma_1}{x},\tfrac{1}{2}\right] & \mbox{if } x>\gamma_1,\\
		\end{cases}
		$$
		and
		$$\partial_{\gamma_1} f_{2}(y)=\partial_{\gamma_1}\left(\tfrac{3}{2}|y|\right)=\tfrac{3}{2}\partial_{2\gamma_1}(|y|)=\begin{cases} \left[-\tfrac{3}{2},-\tfrac{3}{2}-\tfrac{3\gamma_1}{y}\right] & \mbox{if } y<-\gamma_1,\\
			\\
			\left[-\tfrac{3}{2},\tfrac{3}{2}\right] & \mbox{if } -\gamma_1 \le y\le\gamma_1,\\
			\\
			\left[\tfrac{3}{2}-\tfrac{3\gamma_1}{y},\tfrac{3}{2} \right] & \mbox{if } y>\gamma_1.
		\end{cases}
		$$
		Hence, at  $\bar x=0, $ $\bar y=0$,
		$\partial_{\gamma_1} f_1(\bar x)=\left[-\tfrac{1}{2},-\tfrac{1}{2}\right]$ and 
		$\partial_{\gamma_1} f_2(\bar y)=\left[-\tfrac{3}{2},-\tfrac{3}{2}\right]$. Consequently,
		$$\partial_{\gamma_1} f(\bar x, \bar y)=\left[-\tfrac{1}{2},-\tfrac{1}{2}\right] \times \left[-\tfrac{3}{2},-\tfrac{3}{2}\right].$$
		On the other hand, by~\cite[Example 2.1]{Hiriart_1989}, one has
		$$N_{\gamma_2}\big((\bar x,  y); {\rm gph}\, G \big)  = \begin{cases} \{(0, 0)\}  & \mbox{if } y>0,\\
			\{ (u,v) \in \Bbb{R}^2 \mid v\le-|u|\}& \mbox{if } y=0,\\
			\emptyset  & \mbox{if } y<0.
		\end{cases}
		$$
		Therefore,	
		\begin{align*}
			&	\bigcap\limits_{\eta >0} 
			\bigcap\limits_{ y \,\in\, S_\eta (\bar x)}
			\bigcup_{(\gamma_1, \gamma_2)\, \in \, \varGamma(\eta+\varepsilon)}
			\bigg\{\xi \in \mathbb{R} \mid (\xi,0) \in \partial_{\gamma_1}f(\bar x,y) \!+\!N_{\gamma_2}\big ((\bar x,  y); {\rm gph}\, G \big ) \bigg\}\\
			&=  \bigg\{\xi \in \mathbb{R} \mid (\xi,0) \in \left[-\tfrac{1}{2},-\tfrac{1}{2}\right] \times \left[-\tfrac{3}{2},-\tfrac{3}{2}\right]+ \{ (u,v) \in \Bbb{R}^2 \mid v\le-|u|\}\bigg\}\\  
			&=\big\{\xi \in \mathbb{R} \mid \xi \in [-2,2] \big\},
		\end{align*}
		which justifies the conclusion of Theorem~\ref{thm_sensitivity_1}.	
	\end{example}


	\section*{Acknowledgments} 
	This research is funded by Vietnam 
National Foundation for Science and Technology Development (NAFOSTED) under grant number 101.01-2023.23.


\begin{thebibliography}{99}
		\bibitem{An_Yao_19}  D. T. V. An, J.-C. Yao: \textit{Differential stability of convex optimization problems with possibly empty solution sets}, J. Optim. Theory Appl. 181 (2019), pp. 126--143.
		
		\bibitem{An_Yen} D. T. V. An, N. D. Yen: \textit{Differential stability of convex optimization problems under inclusion constraints}, Appl Anal. 94 (2015), pp. 108--128.

  \bibitem{An_Markus-Tuyen-20}  D. T. V. An, M. A. Kobis, N. V. Tuyen:  \textit{Differential stability of convex optimization problems under weaker conditions}, Optimization 69 (2020), pp. 385--399.
		
		\bibitem{Bauschke_Moffat_Wang}H. H. Bauschke, S. M. Moffat, X. Wang: \textit{Near equality, near convexity, sums of maximally monotone operators, and averages of firmly nonexpansive mappings}, Math. Program. 139 (2013), pp. 55--70. 
		
		
		\bibitem{Bertsekas_Mitter_71} D. P. Bertsekas, S. K. Mitter:\textit{ Steepest descent for optimization problems with nondifferentiable 			cost functionals}, Proc. 5th Annual Princeton Conference on Information Sciences and Systems, N.J. Princeton, (1971).
		
		\bibitem{Bertsekas_Mitter_73} D. P. Bertsekas, S. K. Mitter:\textit{ A decent numerical method for optimization problems with nondifferentiable cost functionals}, SIAM J. Control 11 (1973), pp. 637--652.
		
		
		\bibitem{Bot_grad_Wanka_06} R. I. Bo\c{t}, S. M. Grad, G. Wanka:\textit{ Almost convex functions: conjugacy and duality}, Lecture Notes in Economics and Mathematical Systems 583 (2006), pp. 101--114. 
		
		\bibitem{Bot_grad_Wanka_07} R. I. Bo\c{t}, S. M. Grad, G.  Wanka: \textit{Fenchel’s duality theorem for nearly convex functions}, J. Optim. Theory Appl. 132 (2007), pp. 509--515. 
		
		\bibitem{Brondsted_Rockafellar} A. Br\o{n}dsted, R. T. Rockafellar: \textit{On the subdifferentiability of convex functions}, Proc. Amer. Math. Soc. 16 (1965), pp. 605--611. 
		
		\bibitem{Dhara-Dutta-12} A. Dhara, J. Dutta: \textit{ Optimality Conditions in Convex Optimization: A Finite-Dimensional View}, CRC Press Taylor \& Francis Group, 2012. 
		
		\bibitem{Hiriart_1982}J.-B. Hiriart-Urruty: {\em $\varepsilon$-subdifferential calculus}, Res. Notes Math. 57 (1982), Pitman, Boston, pp. 43--92. 
		\bibitem{Hiriart_1989}J.-B. Hiriart-Urruty: \textit{From convex optimization to nonconvex optimization. Necessary and sufficient conditions for global optimality}. In: Clarke, F.H., Demyanov, V.F., Giannessi, F. (eds.) Nonsmooth Optimization and Related Topics. Ettore Majorana Internat. Sci. Ser. Phys. Sci., Vol. 43, pp. 219--239. Plenum, New York, 1989.
		
		\bibitem{H_L_1993} J.-B. Hiriart-Urruty, C. Lemar\'{e}chal: \textit{Convex Analysis and Minimization Algorithms I} \& {\it II}, volume 306: Fundamental Principles of Mathematical Sciences. Springer-Verlag, Berlin, 1993.
		
		\bibitem{H_M_S} J.-B. Hiriart-Urruty, M. Moussaoui, A. Seeger, M. Volle: \textit{Subdifferential calculus without qualification conditions, using approximate subdifferentials: A survey}, Nonlinear Anal. 24 (1995), pp. 1727--1754. 
		
		\bibitem{Huy_Nam_Yen_23} N. Q. Huy, N. M. Nam, N. D. Yen: {\em  Nearly convex optimal value functions and some related topics}, (2023); https://arxiv.org/abs/2303.07793
		
		\bibitem{Ioffe_Penot} A. D. Ioffe, J. P. Penot: \textit{Subdifferentials of performance functions and calculus of set-valued mappings}, Serdica Math J. 22 (1996), pp. 359--384.
		
		\bibitem{Jourani} A. Jourani: \textit{Intersection formulae and the marginal function in Banach spaces}, J. Math. Anal. Appl. 192 (1995), pp. 867--891.
		
		\bibitem{Minty_61}G. J. Minty:\textit{ On the maximal domain of a ``monotone" function}, Michigan Math. J. 8 (1961), pp. 135--137.
		
		\bibitem{Moffat_etal_2016}  S. M. Moffat, W. M. Moursi, X. Wang : \textit{Nearly convex sets: fine properties and domains or ranges of subdifferentials of convex functions}, Math. Program.  160 (2016), pp. 193--223.
		
		
		
		\bibitem{Mordukhovich_etal} B. S. Mordukhovich, N. M. Nam, N. D. Yen: \textit{Subgradients of marginal functions in parametric mathematical programming}, Math. Program. 116 (2009), pp. 369--396.
		
		\bibitem{MoussaouiSeeger1994} M. Moussaoui, A. Seeger: \textit{Sensitivity analysis of optimal value functions of convex parametric programs with	possibly empty solution sets}, SIAM J. Optim. 4 (1994), pp. 659--675. 
		
		
		\bibitem{Nam_Thieu_Yen_23} N. M. Nam, N. N. Thieu, N. D. Yen: {\em  Near convexity and generalized differentiation}, J. Convex Anal. Volume 32 (2025).	
		
		\bibitem{Rockafellar_1970} R. T. Rockafellar: \textit{ On the virtual convexity of the domain and range of a nonlinear maximal monotone operator}, Math. Ann. 185 (1970), pp. 81--90. 
		
		\bibitem{Rockafellar_Wet} R. T. Rockafellar, R. J.  -B. Wets: \textit{Variational Analysis, corrected 3rd printing}, Springer, Berlin, 2009.

\bibitem{Tuyen-et-al-20} N. V. Tuyen, Y.-B. Xiao, T. Q. Son: {\em On approximate KKT optimality conditions for cone-constrained vector optimization problems}, J. Nonlinear Convex Anal. 21 (2020), pp. 105--117.

		\bibitem{Zalinescu_2002}C. Z\u{a}linescu: \textit{Convex Analysis in General Vector Spaces}, World Scientific. New Jersey-London-Singapore-Hong Kong, 2002.
		
	\end{thebibliography}
\end{document}